\documentclass[dvips,preprint,authoryear,12pt]{imsart}

\RequirePackage{natbib,amssymb,verbatim,undertilde,mathrsfs}
\RequirePackage[OT1]{fontenc}
\RequirePackage{amsthm,amsmath,natbib}
\RequirePackage{hypernat}
\usepackage{graphicx}
\startlocaldefs
\theoremstyle{plain}
\newtheorem{thm}{Theorem}
\newtheorem{cor}{Corollary}
\newtheorem{lemma}{Lemma}
\newtheorem*{lemma*}{Lemma}
\newtheorem{defn}{Definition}
\newtheorem{prop}{Proposition}
\theoremstyle{remark}

\endlocaldefs

\def\iidsim{\stackrel{\mathrm{\scriptscriptstyle iid}}{\sim}}
\def\Cov{\mbox{Cov}}
\def\iidsim{\stackrel{\mbox{\tiny iid}}{\sim}}

\def\a{\alpha}
\def\E{\mathscr{E}}

\def\z{\mathbf{z}}
\def\dd{\boldsymbol{\delta}}
\def\bth{\boldsymbol{\theta}}
\def\th{\theta}
\def\vv{\mathbf{v}}

\def\bxi{\boldsymbol{\xi}}

\def\tr{\mathrm{tr}}

\def\1{\mathbf{1}}
\def\aa{\boldsymbol{\alpha}}

\def\e{\epsilon}
\def\ee{\boldsymbol{\epsilon}}
\def\b{{\bf b}}
\def\bb{\boldsymbol{\beta}}
\def\d{\delta}

\def\Var{\mathrm{Var}}

\def\g{\gamma}

\def\x{\mathbf{x}}
\def\w{\mathbf{w}}

\def\s{\sigma}
\def\t{\mathbf{t}}

\def\y{\mathbf{y}}

\def\1{\mathbf{1}}

\def\P{\mathscr{P}}

\def\t{\mathbf{t}}
\def\R{\mathbb{R}}

\def\X{\mathbf{X}}
\def\S{\mathit{\Sigma}}

\begin{document}
\begin{frontmatter}
\title{Optimal Estimation and Prediction for Dense Signals in
 High-Dimensional Linear Models \protect}
\runtitle{Dense Signals and High-Dimensional Linear Models}

\begin{aug}
\author{\fnms{Lee} \snm{Dicker}
\ead[label=e1]{ldicker@stat.rutgers.edu}}


\runauthor{L. Dicker}

\affiliation{Rutgers University}

\address{Department of Statistics and Biostatistics \\ Rutgers University \\ 501 Hill Center, 
 110 Frelinghuysen Road \\ Piscataway, NJ 08854 \\
\printead{e1}}
\end{aug}

\begin{keyword}[class=AMS]
\kwd[Primary ]{62J05}
\kwd[; secondary ]{62C20}
\end{keyword}

\begin{keyword}
\kwd{adaptive estimation}
\kwd{asymptotic minimax}
\kwd{non-Gaussian sequence model}
\kwd{oracle estimators}  
\kwd{ridge regression}
\end{keyword}

\begin{abstract}
Estimation and prediction problems for dense 
signals are often framed in terms of minimax problems over highly
symmetric parameter spaces.  In this paper, we study minimax problems
over $\ell^2$-balls for
high-dimensional linear models with Gaussian predictors. We obtain
sharp asymptotics for the minimax risk that are applicable in any asymptotic
setting where the number of predictors diverges and prove that ridge
regression is asymptotically minimax. Adaptive asymptotic minimax
ridge estimators are also identified.  Orthogonal invariance is
heavily exploited throughout the paper and, beyond serving as a
technical tool, provides additional insight into the problems
considered here.  Most of our results
follow from an apparently novel analysis of an equivalent non-Gaussian sequence
model with orthogonally invariant errors.   As with many dense
estimation and prediction problems, the minimax risk studied here has rate $d/n$,
where $d$ is the number of predictors and $n$ is the number of
observations; however, when $d \asymp n$ the minimax risk is influenced by the spectral distribution of the predictors and is notably different from the linear minimax risk
for the Gaussian sequence model \citep{pinsker1980optimal} that often appears in other dense
estimation and prediction problems.  
\end{abstract}

\end{frontmatter}
\section{Introduction}

This paper is about
estimation and prediction problems involving non-sparse (or ``dense'') signals in high-dimensional
linear models.   By contrast, a great deal of recent research into high-dimensional
linear models  has focused on
sparsity.  Though there are many notions of sparsity
(e.g. $\ell^p$-sparsity \citep{abramovich2006adapting}), a vector
$\bb \in \R^d$ is typically considered to be sparse if many of its
coordinates are very close to 0.   Perhaps one of the general
principals that has emerged from the literature on sparse high-dimensional
linear models may be summarized as follows: if the parameter of interest
is sparse, then this can
often be leveraged to develop methods that perform very well, even
when the number of predictors is much larger than the number of observations.  Indeed, powerful theoretical
performance guarantees are available for many methods developed under
this paradigm, provided the parameter of interest is sparse
\citep{candes2007dantzig, bunea2007sparsity, 
  bickel2009simultaneous, zhang2010nearly,
  fan2011nonconcave, rigollet2011exponential}.  Furthermore, in many
applications -- especially in engineering and signal processing -- sparsity assumptions have been repeatedly validated \citep{donoho1995noising,lustig2007sparse,  duarte2008single,
  wright2008robust, erlich2010compressed}.  However, there is
less certainty about the manifestations of sparsity in other important applications where
high-dimensional data is abundant. For example, several recent papers
have questioned the degree of sparsity in modern genomic datasets
(see, for instance, \citep{hall2009feature}, and the references
contained therein -- including \citep{kraft2009genetic, goldstein2009common, hirschhorn2009genomewide} -- and, more
recently, \citep{bansal2010statistical,manolio2010genomewide}).  In
situations like these, sparse methods may be sub-optimal and methods
designed for dense problems may be more appropriate. 

Let $d$ and $n$ denote the number of predictors and observations,
respectively, in a linear regression problem.  In dense estimation and prediction problems, where the parameter of
interest is not assumed to be sparse, $d/n \to 0$ is often required to ensure
consistency.  Indeed, this is the case for the problems considered in
this paper.  In this sense, dense problems are more challenging than
sparse problems, where consistency may be possible when $d/n \to
\infty$.  This lends credence to Friedman et al.'s (2004)  ``bet on
sparsity'' principle for high-dimensional data analysis:

\begin{quote}
Use a procedure that does well in sparse problems, since no procedure does well in dense problems.  
\end{quote}
The ``bet on sparsity'' principle has proven to be very useful, especially in
applications where sparsity prevails, and it may help to explain some of
the recent emphasis on understanding sparse problems.  However, the emergence of
important problems in high-dimensional data analysis where the role of
sparsity is less clear highlights the importance of characterizing and
thoroughly understanding dense problems in high-dimensional data
analysis.  This paper addresses some of these problems.

Minimax problems over highly symmetric parameter spaces have often been equated with
dense estimation problems in many statistical settings
\citep{donoho1994minimax, johnstone2011gaussian}. 
In this paper, we study the minimax risk over $\ell^2$-balls for
high-dimensional linear models with Gaussian predictors.  
We identify several informative,
asymptotically equivalent formulations of the problem and provide a complete asymptotic solution when the number of predictors $d$
grows large.  In particular, we obtain sharp asymptotics for the
minimax risk that are applicable in any asymptotic setting where $d
\to \infty$ and we show that ridge regression estimators \citep{tikhonov1943stability,
  hoerl1970ridge} are asymptotically minimax.  Adaptive asymptotic minimax
ridge estimators are also discussed.  Our results follow from
carefully analyzing an equivalent non-Gaussian sequence model with
orthogonally invariant errors and the novel use of two
classical tools -- Brown's identity \citep{brown1971admissible} and
Stam's inequality \citep{stam1959some} -- to relate this sequence
model to the Gaussian sequence model with iid errors.  The results in
this paper
share some similarities with those found in
\citep{goldenshluger2001adaptive, goldenshluger2003optimal}, which
address minimax prediction over $\ell^2$-ellipsoids.  However, the
implications of our results and the methods used to prove them differ
substantially from Goldenshluger and Tsybakov's (this is discussed in more detail in Sections
2.2-2.3 below).  

\section{Background and preliminaries}

\subsection{Statistical setting}

Suppose that the observed data consists of outcomes $y_1,...,y_n\in \R$ and
$d$-dimensional predictors $\x_1,...,\x_n \in \R^d$.  The outcomes and
predictors follow a linear model and are related via the equation 
\begin{equation}\label{lm}
y_i = \x_i^T\bb + \e_i, \ \ i = 1,...,n,
\end{equation}
where $\bb = (\beta_1,...,\beta_d)^T \in \R^d$ is an unknown parameter
vector (also referred to as ``the signal'') and $\e_1,...,\e_n$ are
unobserved errors.  To simplify notation, let $\y = (y_1,...,y_n)^T
\in \R^n$, $X = (\x_1,...,\x_n)^T$, and $\ee = (\e_1,...,\e_n)^T$.
Then (\ref{lm}) may be rewritten as $\y = X\bb + \ee$.  In many
high-dimensional settings it is natural to consider the predictors
$\x_i$ to be random.  In this paper, we assume that 
\begin{equation}\label{norm}
\x_1,...,\x_n \iidsim
N(0,I) \mbox{ and } \e_1,...,\e_n \iidsim N(0,1)
\end{equation}
are independent, where $I = I_d$ is
the $d \times d$ identity matrix.  These
distributional assumptions impose significant additional structure on the
linear model (\ref{lm}).  However, similar models have been studied previously \citep{stein1960multiple, baranchik1973inadmissibility, breiman1983many, brown1990ancillarity, leeb2009conditional} and we believe that the insights imparted by the resulting simplifications are worthwhile.
For the results in this paper, perhaps the most noteworthy simplifying consequence of the normality
assumption (\ref{norm}) is that the distributions of $\X$ and $\ee$
are invariant under orthogonal transformations.  

We point
out that the assumption $E(\x_i) = 0$ (which is implicit in
(\ref{norm})) is not particularly limiting:
if $E(\x_i) \neq 0$, then we can reduce to the mean 0 case by
centering and decorrelating the data. If $\Var(\e_i) = \s^2 \neq 1$
and $\s^2$ is known, then this can easily be reduced to the case where
$\Var(\e_i) = 1$.  If $\s^2$ is unknown and $d < n$, then $\s^2$ can be
effectively estimated and one can reduce to the case where $\Var(\e_i)
= 1$ \citep{dicker2012dense}.  We conjecture that $\s^2$ can be effectively
estimated when $d > n$, provided $\sup d/n < \infty$
(for sparse $\bb$, \cite{sun2011scaled} and \cite{fan2012variance}
have shown that $\s^2$ can be estimated when $d \gg n$).  
\cite{dicker2012dense} has discussed the implications if $\Cov(\x_i) = \S \neq
I$.  Essentially, when the emphasis is prediction and non-sparse signals, if a
norm-consistent estimator for $\Cov(\x_i) = \S$ is available, then it is possible
to reduce to the case where $\Cov(\x_i) = I$; if a norm-consistent
estimator is not available, then limitations entail, however, these
limitations may not be overly restrictive (this is discussed further in
Section 3.2 below).  

Let $||\cdot || = ||\cdot||_2$ denote the $\ell^2$-norm.  In this
paper we study the performance of estimators $\hat{\bb}$ for $\bb$ with
respect to the risk function
\begin{equation}\label{risk}
R(\hat{\bb},\bb) = R_{d,n}(\hat{\bb},\bb) = E_{\bb}||\hat{\bb} - \bb||^2 ,
\end{equation}
where the
expectation is taken over $(\ee,X)$ and the subscript $\bb$ in
$E_{\bb}$ indicates that $\y = X\bb + \ee$ (below, for expectations that do
not involve $\y$, we will often omit this subscript).  We emphasize that the
expectation in (\ref{risk}) is taken over the predictors $X$ as well
as the errors $\ee$, i.e. it is {\em not} conditional on $X$.  The risk
$R(\hat{\bb},\bb)$ is a measure of estimation error.  However,  
it can also be interpreted as the unconditional out-of-sample prediction
error (predictive risk) associated with the
estimator $\hat{\bb}$
\citep{stein1960multiple, breiman1983many, leeb2009conditional}.  

\subsection{Dense signals, sparse signals, and ellipsoids} 

Let $B(c) = B_d(c) = \{\bb \in \R^d; \ ||\bb|| \leq c\}$ denote the
$\ell^2$-ball of radius $c \geq 0$.  Though a given signal $\bb \in \R^d$ is often
considered to be dense if it has many nonzero entries, when studying broader
properties of dense signals and dense estimators it is common to consider
minimax problems over highly symmetric, convex (or loss-convex
\citep{donoho1994minimax}) parameter spaces.  Following this approach,
one of the primary quantities that we use as a benchmark for evaluating estimators and determining performance limits in dense
estimation problems  is the
minimax risk over $B(c)$: 
\begin{equation}\label{rb}
R^{(b)}(c) = R_{d,n}^{(b)}(c) = \inf_{\hat{\bb}} \sup_{\bb \in B(c)} R(\hat{\bb},\bb).
\end{equation}
The infimum on the right-hand side in (\ref{rb}) is taken over all measurable
 estimators $\hat{\bb}$ and the superscript ``$b$'' in $R^{(b)}(c)$ indicates
 that the relevant parameter space is the $\ell^2$-ball.

A basic consequence of the results in this paper is $R^{(b)}(c) \asymp
d/n$.  Thus, one must have $d/n \to 0$ in order to ensure consistent
estimation over $B(c)$.  This is a well-known feature of dense estimation problems and, as mentioned in
Section 1, contrasts with many
 results on sparse estimation that imply $\bb$ may be
consistently estimated when $d/n \to \infty$. However, the sparsity
conditions on $\bb$ that are required for these results may not hold in general and our motivating interest lies
precisely in such situations. 
In this paper we derive sharp asymptotics for $R^{(b)}(c)$ and related
quantities in settings where $d/n \to 0$, $d/n \to
\rho  \in (0,\infty)$, and $d/n \to \infty$ (we assume that $d \to
\infty$ throughout).   Though
consistent estimation is only guaranteed when $d/n \to 0$, there are important situations where one
might hope to analyze high-dimensional datasets with $d/n$
substantially larger than 0, even if there is
little reason to believe that sparsity assumptions are valid.  The results in
this paper provide detailed information that may be useful in
situations like these.  

In addition to sparse estimation problems, minimax rates faster than $d/n$ have also been obtained for minimax
problems over $\ell^2$-ellipsoids, which have been studied extensively in
situations similar to those considered here
 \citep{pinsker1980optimal, cavalier2002sharp, goldenshluger2001adaptive,
  goldenshluger2003optimal}.  Much of this work has been motivated by
problems in nonparametric function estimation.  The results in
this paper are related to many of these existing results, however, there are important
differences -- both in their implications and the techniques used to
prove them.  Goldenshluger and Tsybakov's
(2001, 2003) work may be most closely related to ours.  
Define the $\ell^2$-ellipsoid $B(c,\aa) = \{\bb \in \R^d;
\ \sum_{i = 1}^n \a_i\beta_i^2 \leq c^2\}$, with $\aa = (\a_1,...,\a_d)^T
\in \R^d$, $0 \leq \a_1 \leq \cdots \leq \a_d$.   Goldenshluger
and Tsybakov studied minimax problems over
$\ell^2$-ellipsoids for a linear model with random predictors similar
to the model considered here (in fact, Goldenshluger and Tsybakov's results
apply to infinite-dimensional non-Gaussian $\x_i$, though $\x_i$ are required to have Gaussian
tails and independent coordinates).  They identified asymptotically
minimax estimators over $B(c,\aa)$ and adaptive asymptotically minimax
estimators and showed that the minimax rate may be substantially faster than $d/n$.
However, their results also require the axes of $B(c,\aa)$ to decay
rapidly (i.e. $a_d/c \to \infty$ quickly) and do
not apply to $\ell^2$-balls $B(c) = B(c,(1,...,1)^T)$ unless $d/n \to
0$.  Though
these decay conditions are natural for many inverse problems in
nonparametric function estimation, they drive the improved minimax
rates obtained by
Goldenshluger and Tsybakov and may be overly
restrictive in other settings, such as the genomics applications
discussed in Section 1 above.

\subsection{The sequence model}

Minimax problems over restricted parameter spaces have
been studied extensively in the context of the sequence model.  In the
sequence model, given an index set $J$, 
\begin{equation}\label{seq}
z_j = \th_j + \d_j, \ \ j \in J,
\end{equation}  
are observed, $\bth = (\th_j)_{j \in J}$ is
an unknown parameter, and $\dd = (\d_j)_{j \in J}$ is a random error.
The sequence model is extremely
flexible, and many existing results about the Gaussian sequence model
(where the coordinates of $\dd$ are iid Gaussian random variables) have implications for
high-dimensional linear models \citep{pinsker1980optimal,
   cavalier2002sharp}.  However, these results tend to apply
 in linear models where one conditions on the predictors, as opposed
to random predictor models like the one considered here.   

In order to prove the main result in this paper (Theorem 1), we study
a sequence model with non-Gaussian orthogonally invariant errors that is equivalent to the linear model (\ref{lm}).
\cite{goldenshluger2001adaptive} also studied a non-Gaussian sequence
model  that derives from a high-dimensional linear model with random
 predictors, but their results have limitations in settings where $d/n
 \to \rho > 0$, as discussed in Section 2.2 above.  In our analysis,
 orthogonal invariance is heavily exploited to obtain precise results
 in any asymptotic setting where $d \to \infty$.
 This appears to be a key difference between our analysis and Goldenshluger and Tsybakov's.

\subsection{Minimax problems over
  $\ell^2$-spheres and orthogonal equivariance}

Define the $\ell^2$-sphere of radius $c$, $S(c) = S_d(c) = \{\bb \in
\R^d; \ ||\bb|| = c\}$.  Though it is common in dense estimation
problems to study the minimax risk over $\ell^2$-balls $R^{(b)}(c)$,
which is one of the primary objects of study here, we find it convenient and informative to consider a closely related quantity, the minimax risk over $S(c)$,
\[
R^{(s)}(c) = R^{(s)}_{d,n}(c) = \inf_{\hat{\bb}} \sup_{\bb \in S(c)} R(\hat{\bb},\bb)
\]
(the superscript ``$s$'' in $R^{(s)}(c)$ stands for ``sphere'').  For
our purposes, the primary significance of considering $\ell^2$-spheres
comes from connections with orthogonal invariance and equivariance.  Let $O(d)$
denote the group of $d \times d$ orthogonal matrices.  

\vspace{.1in} 
{\em Definition 1.} An estimator $\hat{\bb} = \hat{\bb}(\y,X)$ for $\bb$ is {\em orthogonally equivariant} if 
\begin{equation}\label{orthequiv}
U^T\hat{\bb}(\y,X) = \hat{\bb}(\y,XU)
\end{equation}
for all $U \in O(d)$.  \hfill $\Box$

\vspace{.1in}

Orthogonally equivariant estimators are compatible with orthogonal
transformations of the predictor basis.  They may be appropriate when
there is little information carried in the given predictor basis
vis-\`a-vis the outcome; by contrast, knowledge about sparsity
is exactly one such piece of information.  Indeed, sparsity
assumptions generally imply that in the given basis some predictors are significantly
more influential than others.  Sparse estimators
attempt to take advantage of this to improve
performance and are typically not orthogonally equivariant.   

The concept of equivariance
plays an important role in statistical
decision theory (e.g. \citep{berger1985statistical}, Chapter 6).
However, it seems to have received relatively little attention in the
context of linear models.  Significant aspects of equivariance
include: (i) in certain cases, one can show that
it suffices to consider equivariant estimators when studying minimax
problems and (ii) equivariance may provide a
convenient means for identifying minimax estimators.  This is
basically the content of the Hunt-Stein theorem and both of these
features prevail in the present circumstances.  To make this more
precise, define the class of equivariant estimators
\[
\E = \E(n,d) = \{\hat{\bb}; \hat{\bb} \mbox{ is an
  orthogonally equivariant estimator for } \bb\}
\]
and define
\[
R^{(e)}(\bb) = R^{(e)}_{d,n}(\bb) = \inf_{\hat{\bb} \in \E} R(\hat{\bb},\bb).  
\]
Additionally, let $\pi_c$
denote the uniform measure on $S(c)$ and let 
\[
\hat{\bb}_{unif}(c) = \hat{\bb}_{unif}(\y,X;c) = E_{\pi_c}(\bb|\y,X)
\]
be the posterior mean of $\bb$ under the assumption that $\bb \sim
\pi_c$ is independent of $(\ee,X)$.  Since, for $U \in O(d)$,
\[
U^T\hat{\bb}_{unif}(\y,X;c)  =  E_{\pi_c}(U^T\bb|\y,X) 
 = 
 E_{\pi_c}(\bb|\y,XU)  
 = \hat{\bb}_{unif}(\y,XU;c),
\]
it follows that $\hat{\bb}_{unif}(c) \in \E$. The next result
follows directly from the Hunt-Stein theorem and its proof is omitted.  

\begin{prop}  Suppose that $||\bb|| = c$.  Then
\begin{equation} \label{prop1a}
R^{(s)}(c) = R^{(e)}(\bb) = R\{\hat{\bb}_{unif}(c),\bb\}.
\end{equation}
Furthermore, if $\hat{\bb} \in \E$, then $R(\hat{\bb},\bb)$
depends on $\bb$ only through $c$.  
\end{prop}

In a sense, Proposition 1 completely solves the minimax problem over
$S(c)$.  On the other hand, the minimax estimator
$\hat{\bb}_{unif}(c)$ is challenging to compute and it
is desirable to identify good estimators that have a simpler form.
Moreover, though $\hat{\bb}_{unif}(c)$ solves the minimax problem over
$S(c)$, it is unclear how $R^{(s)}(c)$ relates to the minimax risk
over $\ell^2$-balls, which is a more commonly studied quantity in dense
estimation problems.  Finally, the minimax estimator
$\hat{\bb}_{unif}(c)$ depends on $c = ||\bb||$, which is
typically unknown in practice.  All of these issues must be addressed
in order to identify practical estimators that perform well in dense
problems for 
high-dimensional linear models.
This is accomplished below, where we show: (i) a linear estimator (ridge regression)
is asymptotically equivalent to $\hat{\bb}_{unif}(c)$, (ii)
$R^{(b)}(c) \sim R^{(s)}(c)$ (i.e. $R^{(b)}(c)/R^{(s)}(c) \to 1$), and (iii) under
certain conditions $c = ||\bb||$ may be effectively estimated.
Similar results have been obtained for the Gaussian sequence model with iid errors
\citep{marchand1993estimation, beran1996stein}.  Our results
rely on an inequality of Marchand's (Proposition 11 below) and extend
Marchand's and Beran's results to linear models with Gaussian
predictors.  

Proposition 1 and the related discussion imply that equivariant
estimators have certain nice properties and are closely linked with dense estimation problems.  On the
other hand, the
next result describes some of the limitations of orthogonally
equivariant estimators when $d > n$ and is indicative of some of the
challenges inherent in dense estimation problems beyond the
consistency requirement $d/n \to 0$.

\begin{lemma} Suppose that $\hat{\bb} =
\hat{\bb}(\y,X) \in \E$.  Then $\hat{\bb}$ is orthogonal to the
null-space of $X$.  
\end{lemma}

\begin{proof}  Suppose that $\mathrm{rank}(X) = r < d$ and let $X = UDV^T$ be the singular value decomposition of
$X$, where $U \in O(n)$, $V \in O(d)$, 
\[
D = \left(\begin{array}{cc} D_0 &  0 \\ 0 & 0 \end{array}\right)
\]
is an
$n \times d$ matrix, and $D_0$ is an $r \times r$ diagonal matrix with
rank $r$.
Let $V_0$ denote the first $r$ columns of $V$ and let $V_1$ denote the
remaining $d-r$ columns of $V$.  Finally, suppose that $W_1 \in
O(d-r)$ and let
\[
W = \left(\begin{array}{cc} I & 0 \\ 0 & W_1 \end{array}\right) \in
O(d).  
\]
Then the null space of $X$ is equal to the column space of $V_1$ and it
suffices to show that $V_1^T\hat{\bb} = 0$.  By equivariance,
\begin{equation}\label{lemma1proof0}
\hat{\bb} = VW\hat{\bb}(\y,XVW) = VW\hat{\bb}(\y,UD). 
\end{equation}
Thus,
\begin{equation}\label{lemma1proof1}
V_1^T\hat{\bb} = V_1^TVW\hat{\bb}(\y,UD) = \left(0 \  \ W\right)\hat{\bb}(\y,UD).
\end{equation}
Since $\hat{\bb}(\y,UD)$ does not depend on $W$ and
(\ref{lemma1proof1}) holds for all $W \in O(d-r)$, it follows that
$V_1^T\hat{\bb} = 0$, as was to be shown.
\end{proof}

  Lemma 1 is a non-estimability result for
orthogonally equivariant estimators.  It will be used in
Sections 3.3 and 6 below.   

\subsection{Linear estimators: Ridge regression}

Linear estimators play an important role in dense estimation
problems in many statistical settings. Fundamental references include
\citep{stein1955inadmissibility, james1961estimation, pinsker1980optimal}.  \cite{pinsker1980optimal} showed that under certain
conditions, linear estimators in the Gaussian sequence model are asymptotically minimax over $\ell^2$-ellipsoids.   In the linear model, linear
estimators have the form $\hat{\bb} = A\y$, where $A$ is a data-dependent $d
\times n$ matrix, and they 
are convenient because of their simplicity. Define the ridge regression estimator
\[
\hat{\bb}_r(c) = (X^TX + d/c^2 I)^{-1}X^T\y, \ \ c \in [0,\infty].
\]
By convention, we take $\hat{\bb}_r(0) = 0$ and $\hat{\bb}_r(\infty) = \hat{\bb}_{ols} =
(X^TX)^{-1}X^T\y$ to be the ordinary least squares (OLS) estimator.
Furthermore, throughout the paper, if a matrix $A$ is not invertible, then $A^{-1}$
is taken to be its Moore-Penrose pseudoinverse (thus, the OLS
estimator is defined for all $d,n$).  
Clearly, $\hat{\bb}_r(c)$ is a linear estimator.  Furthermore, it is
easy to check that
$\hat{\bb}_r(c) \in \E$.  

\cite{dicker2012dense} studied finite
sample and asymptotic properties of 
$R\{\hat{\bb}_r(c),\bb\}$.   Some of these properties will be used in
this paper and are summarized presently.  

\subsubsection{Oracle estimators} 
Define the oracle ridge
regression estimator
\[
\hat{\bb}_r^* = \hat{\bb}_r(||\bb||).
\]  
This estimator is called an oracle estimator
because it depends on $||\bb||$, which is
typically unknown.   Proposition 5 of \citep{dicker2012dense} implies
\begin{equation}\label{ridgerisk}
R(\hat{\bb}_r^*,\bb) = \inf_{c \in [0,\infty]}
R\{\hat{\bb}_r(c),\bb\} = E\tr(X^TX + d/||\bb||^2I)^{-1}
\end{equation}
and, furthermore,
\begin{equation}\label{ridgemono}
R\{\hat{\bb}_r(||\bb||),\bb_0\}  \leq R\{\hat{\bb}_r(||\bb||),\bb\},
\mbox{ if }  ||\bb_0|| \leq ||\bb||.
\end{equation}
The next result gives an expression for the asymptotic predictive risk of
$\hat{\bb}_r^*$.  Its proof relies heavily on properties of the
Mar\v{c}enko-Pastur distribution \citep{marcenko1967distribution, bai1993convergence}.  

\begin{prop}[Proposition 8 from \citep{dicker2012dense}]
Suppose that $0 < \rho^- \leq d/n \leq \rho^+ < \infty$ for some fixed
constants $\rho^-,\rho^+ \in \R$ and define
\[
r_{>0}(\rho,c)  =  \frac{1}{2\rho}\left[c^2(\rho - 1) - \rho + \sqrt{\{c^2(\rho - 1) - \rho\}^2 + 4c^2\rho^2}\right].
\] 
\begin{itemize}
\item[(a)] If $0 < \rho^- < \rho^+ < 1$ or $1 < \rho^- < \rho^+ < \infty$ and $n - d > 5$, then
\[
\left| R(\hat{\bb}_r^*,\bb) - r_{>0}(d/n,||\bb||)\right| =
O\left(\frac{||\bb||^2}{ ||\bb||^2 + 1}n^{-1/4}\right).
\]
\item[(b)] If $0 < \rho^- < 1 < \rho^+ < \infty$, then
\[
\left| R(\hat{\bb}_r^*,\bb) - r_{>0}(d/n,||\bb||)\right| = O(||\bb||^2 n^{-5/48}).
\]
\end{itemize} 
\end{prop}

Notice that Proposition 2 implies the asymptotic predictive risk
of $\hat{\bb}_r^*$ is non-vanishing if $d/n \to \rho > 0$.  
The main results in this paper are essentially asymptotic optimality
results for $\hat{\bb}_r^*$.   In particular,
we show that $\hat{\bb}_r^*$ is asymptotically minimax
over $\ell^2$-balls and $\ell^2$-spheres, and 
asymptotically optimal among the class of orthogonally equivariant
estimators. Combined with Propositions 2-3, these results immediately yield
sharp asymptotic for $R^{(b)}(c)$, $R^{(s)}(c)$, and $R^{(e)}(\bb)$.  

Taking a Bayesian point-of-view, our optimality results for $\hat{\bb}_r^*$ are not surprising.  Indeed, in Section 2.3 we observed that if $||\bb|| = c$, then $\hat{\bb}_{unif}(c) = E_{\pi_c}(\bb|\y,X)$ is minimax
over $S(c)$ and is optimal among orthogonally equivariant estimators
for $\bb$.  On the other hand, if $||\bb|| = c$, then the oracle ridge estimator
$\hat{\bb}_r^* = \hat{\bb}_r(c) = E_{N(0,c^2/dI)}(\bb|\y,X)$
may be interpreted
as the posterior mean of $\bb$
under the assumption that $\bb \sim N\{0,(c^2/d)I\}$ is
independent of $\ee$ and $X$.  Furthermore, if $d$ is
large, then the normal distribution $N\{0,(c^2/d)I\}$ is ``close'' to
$\pi_c$ (there is an enormous body of literature that makes this idea
more precise -- \cite{diaconis1987dozen} attribute early work to
\cite{borel1914introduction} and \cite{levy1922lecons}).  Thus, it is
reasonable to expect that $\hat{\bb}_{unif}(c) \approx \hat{\bb}_r(c) $
and that, asymptotically, the oracle ridge estimator shares the
optimality properties of $\hat{\bb}_{unif}(c)$, which is indeed the
case.  

\subsubsection{Adaptive estimators} 

Adaptive ridge estimators will also be discussed in this paper.  As mentioned above, $||\bb||$ is typically unknown; hence,
$\hat{\bb}_r^*$ is typically  non-implementable.  However,
$\hat{\bb}_r^*$ may be approximated by an adaptive estimator where
$||\bb||$ is replaced with an estimate -- this estimator ``adapts'' to
the unknown quantity $||\bb||$.  Define 
\[
\widehat{||\bb||}^2 = \max\left\{\frac{||\y||^2}{n} - 1,0\right\}
\]
and define the adaptive ridge estimator
\begin{equation}\label{adapt}
\check{\bb}_r^* = \hat{\bb}_r(\widehat{||\bb||}).
\end{equation}
Note that $\widehat{||\bb||}^2$ is a consistent estimator of
$||\bb||^2$, as $n \to \infty$.  

\begin{prop} Suppose that $0 < \rho^- \leq d/n \leq \rho^+ < 1$ for
  some fixed constants $\rho^-,\rho^+ \in \R$.   If $n - d > 5$, then 
\[
 \left|R(\hat{\bb}_r^*,\bb) -
  R(\check{\bb}_r^*,\bb)\right| = O\left(\frac{1}{||\bb||^2 +
  1}n^{-1/2}\right).
\]
\end{prop}

The proof of Proposition 3 is nearly identical to the proof of
Proposition 10 from \citep{dicker2012dense} and is omitted.  Proposition 3 implies that if $d/n \to \rho \in (0,1)$, then the
adaptive ridge estimator has nearly the same asymptotic risk as the
oracle ridge estimator.  Note the restriction $d/n < 1$ in
Proposition 3.  This restriction also appears in
\citep{dicker2012dense}, where $\Var(\e_i) = \s^2$ is unknown and the
signal-to-noise ratio $||\bb||^2/\s^2$, as opposed to $||\bb||^2$, is
the quantity that must be estimated to obtain an adaptive ridge
estimator; in this context, $d/n < 1$ is a fairly natural condition
for estimating $\s^2$.  It is possible to extend Proposition 3 to
settings where $d/n > 1$.  However, if $d/n > 1$, then the
corresponding error term in Proposition 3 is no longer uniformly
bounded in $||\bb||^2$.  Additionally, notice that Proposition 3 does not apply to
settings where $d/n \to 0$.  A more careful analysis may lead to
extensions in this direction as well.  Since adaptive estimation is not the main
focus of this article, these issues are not pursued further here; however,
future research into these issues may prove interesting.

\subsection{Outline of the paper}

The main results of the paper are stated in Section 3.  Most of these
results follow from Theorem 1, which is stated
at the beginning of the section.  The remainder
of the paper is devoted to proving Theorem 1.  In
Section 4, the equivalence between the linear model and the sequence
model is formalized.  The first part of Theorem 1, which applies to
the setting where $d \leq n$, is proved in Section 5.  This part of
the proof involves
converting error bounds for the Gaussian sequence model with iid
errors into useful bounds for the relevant non-Gaussian sequence model.  The
second part of Theorem 1 ($d > n$) is proved in Section 6.  When $d >
n$, $X^TX$ does not have full rank.  The major steps in the proof for $d > n$
involve reducing the problem to a full rank problem.  

\section{Main results}

The results in this section are presented in terms of the linear model. However, most have
equivalent formulations in terms of the sequence model introduced in
Section 4 below.  

\begin{thm} Suppose that $n > 2$ and let $s_1 \geq \cdots
\geq s_{d \wedge n} > 0$ denote the nonzero (with probability 1) eigenvalues of $(X^TX)^{-1}$.  
\begin{itemize}
\item[(a)] If $d \leq n$, then
\[
\left|R(\hat{\bb}_r^*,\bb) - R^{(e)}(\bb)\right| \leq 
\frac{1}{d}E\left\{\frac{s_1}{s_d}\tr\left(X^TX +
    \frac{d}{||\bb||^2}I\right)^{-1}\right\} 
\]
\item[(b)] If $d > n$, then 
\begin{eqnarray*}
\left|R(\hat{\bb}_r^*,\bb) - R^{(e)}(\bb)\right|&\leq & 
\frac{1}{n}E\left\{\frac{s_1}{s_n}\tr\left(XX^T +
    \frac{d}{||\bb||^2}I\right)^{-1}\right\} \\
&& \ +
2\frac{d-n}{n-2}\frac{1}{||\bb||^2}E\tr\left(XX^T
    + \frac{d}{||\bb||^2}I\right)^{-2}.
\end{eqnarray*}
\end{itemize}\end{thm}

From (\ref{ridgerisk}) and Proposition 1, it is clear that
$R(\hat{\bb}_r^*,\bb)$ and $R^{(e)}(\bb)$ are finite.  Moreover, basic
properties of the Wishart and inverse Wishart distributions imply that
the upper bounds in Theorem 1 are finite, provided $|n-d| >
1$; when $|n - d| \leq 1$, these bounds are infinite.  However, if $|n-d| \leq 1$, then the inequalities 
$R_{d,n}(\hat{\bb}_r^*,\bb) \leq R_{d,n-1}(\hat{\bb}_r^*,\bb)$
 and $R^{(e)}_{d,n}(\bb) \leq R^{(e)}_{d,n-1}(\bb)$ may be combined
 with Theorem 1 (b) to obtain nontrivial bounds.  

In what remains of
 this section, we discuss some of the consequences of Theorem 1 and related
 results in three asymptotic settings: $d/n \to 0$ (with $d \to \infty$, as well), $d/n \to \rho \in (0,\infty)$, and $d/n \to \infty$.  

\subsection{$d/n \to 0$} 

\begin{prop} Define 
\[
r_0(\rho,c) = \frac{c^2\rho}{c^2 + \rho}.
\]
If $d/n \to 0$ and $d \to \infty$, then 
\[
R(\hat{\bb}_r^*,\bb) \sim R^{(e)}(\bb) \sim R^{(s)}(||\bb||) \sim
R^{(b)}(||\bb||) \sim r_0(d/n,||\bb||)
\]
uniformly for $\bb \in \R^d$.  
\end{prop}

\begin{proof} If $d + 1< n$, then (\ref{ridgerisk}) and Jensen's inequality imply that
\begin{equation}\label{jensen}
\frac{d/n}{1 + d/(n||\bb||^2)} \leq R(\hat{\bb}_r^*,\bb) \leq
\frac{d/n}{1 - (d + 1)/n + d/(n||\bb||^2)}.
\end{equation}
It follows that $R(\hat{\bb}_r^*,\bb) \sim
r_0(d/n,||\bb||)$.  By Theorem 1, in order
to prove 
\begin{equation}\label{prop2proof1}
R^{(e)}(\bb) \sim r_0(d/n,||\bb||),
\end{equation}
it suffices to show that 
\[
\frac{1}{d}E\left\{\frac{s_1}{s_d}\tr\left(X^TX +
    d/||\bb||^2I\right)^{-1}\right\}  = o\{r_0(d/n,||\bb||)\}.
\]
But this is clear:
\begin{eqnarray}\label{prop2proof2}
\frac{1}{d}E\left\{\frac{s_1}{s_d}\tr\left(X^TX +
    \frac{d}{||\bb||^2}I\right)^{-1}\right\}  & \leq & 
\frac{||\bb||^2}{d(||\bb||^2 + d/n)}\\ \nonumber && \ \ \cdot E\left\{\frac{s_1}{s_d}\left(d s_1 +
  \frac{d}{n}\right) \right\} \\ \nonumber
& = & O\left\{\frac{1}{d}r_0(d/n,||\bb||)\right\} \\
& = & o\{r_0(d/n,||\bb||)\}, \nonumber
\end{eqnarray}
where we have used the facts $E(s_1^k) = O(n^{-k})$ and $E(s_d^{-k}) =
O(n^{k})$ (Lemma A2, \citep{dicker2012dense}).  
Thus, (\ref{prop2proof1}).  Since $R^{(s)}(||\bb||) = R^{(e)}(\bb)$,
all that is left is to prove is $R^{(b)}(||\bb||) \sim
R^{(s)}(||\bb||)$.  This follows because
\begin{equation}\label{prop2proof3}
R^{(s)}(||\bb||) \leq R^{(b)}(||\bb||) \leq R(\hat{\bb}_r^*,\bb)
\sim R^{(s)}(||\bb||),
\end{equation}
where we have used (\ref{ridgemono}) to obtain the second inequality.  
\end{proof}

The asymptotic
risk $r_0(\rho,c)$ appears frequently in the analysis of linear
estimators for the Gaussian sequence model \citep{pinsker1980optimal} and is often referred to as
the ``linear minimax risk.''  The condition $d \to \infty$ in
Proposition 4 is important
because it drives the approximation $\pi_c \approx N(0,c^2/dI)$, which
enables us to conclude $R^{(e)}(\bb) \sim R(\hat{\bb}_r^*,\bb)$ (re:
the discussion at the end of Section 2.4).  Notice that $\lim_{d/n \to
  0} r_0(\rho,c) = 0$.  Thus, the minimax risk vanishes when $d/n \to 0$.  

Proposition 4 implies that the ridge estimator $\hat{\bb}_r^*$ is
asymptotically minimax if $d/n \to 0$ and $d \to \infty$.  On the other hand, other
simple linear estimators are also asymptotically minimax in this
setting.  Define the estimator 
\[
\hat{\bb}_{scal}^* = \frac{1 - (d + 1)/n}{1 - (d+1)/n +
  d/(n||\bb||^2)}\hat{\bb}_{ols}.
\]  Note that
$\hat{\bb}_{scal}^*$ is a scalar multiple of the OLS estimator and
that $\hat{\bb}_{scal}^*$ is defined for all $d,n$ since
$\hat{\bb}_{ols}$ is defined using pseudoinverses.  Various versions of $\hat{\bb}_{scal}$ have been studied
previously \citep{stein1960multiple, baranchik1973inadmissibility,
  brown1990ancillarity}.  \cite{dicker2012dense} showed that if $d + 1 <
n$, then
\begin{eqnarray} \label{ridgejs}
R(\hat{\bb}_r^*,\bb) & \leq & R(\hat{\bb}_{scal}^*,\bb) = \frac{d/n}{1 - (d+1)/n +
  d/(n||\bb||^2)} \\ \nonumber
& \leq & R(\hat{\bb}_{ols},\bb) = \frac{d/n}{1 - (d + 1)/n}.
\end{eqnarray}
The following corollary to Proposition 4 follows immediately.

\begin{cor}
\begin{itemize}
\item[(a)] If $d/n \to 0$ and $d \to \infty$, then
\[
R(\hat{\bb}_{scal}^*,\bb) \sim R^{(b)}(||\bb||)
\]
uniformly for $\bb \in \R^d$.  
\item[(b)] If $d/n \to 0$, $d \to \infty$, and $d/(n||\bb||^2) \to s
  \geq 0$,  then 
\[
\frac{R(\hat{\bb}_{ols},\bb)}{R^{(b)}(||\bb||)} \to 1 + s.
\]
\end{itemize} \end{cor}

In other words, if $d/n\to 0$ and $d \to \infty$, then
$\hat{\bb}_{scal}$ is asymptotically minimax over $\ell^2$-balls (and,
moreover, asymptotically equivalent to $\hat{\bb}_r^*$).  Furthermore,
the OLS estimator may be asymptotically minimax over $\ell^2$-balls,
but this depends on the magnitude of the signal $\bb$:  If $||\bb||^2$ is
large, then the OLS estimator is asymptotically minimax; if $||\bb||^2$ is
small, then it is not.  

\subsection{$d/n \to \rho \in (0,\infty)$}

The setting where $d/n \to \rho \in (0,\infty)$ may be the
most interesting one for the dense estimation problems considered
here.  The minimax risk is non-vanishing in this setting; however, informative
closed form expressions for the limiting minimax risk are available.  Moreover, differences
between the linear estimators $\hat{\bb}_{scal}^*$ and $\hat{\bb}_r^*$
which are insignificant when $d/n \to 0$ become pronounced when $d/n
\to \rho \in (0,\infty)$.  These differences are largely attributable to the spectral
distribution of $n^{-1}X^TX$, which is asymptotically trivial if $d/n
\to 0$ and converges to the Mar\v{c}enko-Pastur distribution
\citep{marcenko1967distribution} if $d/n \to
\rho \in(0,\infty)$.  

\begin{prop} Suppose that $\rho \in (0,\infty)$ and let
$R^*(\bb)$ denote any of $R(\hat{\bb}_r^*,\bb)$, $R^{(e)}(\bb)$,
$R^{(s)}(||\bb||)$, or $R^{(b)}(\bb)$.  If $\rho
\neq 1$, then 
\begin{equation}\label{prop4a}
\lim_{d/n \to \rho} \sup_{\bb \in \R^d} \left|R^*(\bb) -
  r_{>0}(d/n,||\bb||)\right| = 0,
\end{equation}
where $r_{>0}(\rho,c)$ is defined in Proposition 2 above.  
Furthermore, as $d/n \to \rho$, 
\begin{equation}\label{prop4b}
R(\hat{\bb}_r^*,\bb) \sim R^{(e)}(\bb) \sim R^{(s)}(||\bb||) \sim
R^{(b)}(||\bb||) \sim r_{>0}(d/n,||\bb||).
\end{equation}
If $\rho \neq 1$, then the implied convergence in (\ref{prop4b}) holds
uniformly for $\bb \in \R^d$; if $\rho = 1$, then the convergence is
uniform over $B(c)$ for any fixed $c \in (0,\infty)$.  
\end{prop}

\begin{proof} Proposition 2 implies that $|R(\hat{\bb}_r^*,\bb) -
r_{>0}(d/n,||\bb||)| \to 0$ and $R(\hat{\bb}_r^*,\bb) \sim r_{>0}(d/n,||\bb||)$,
with the appropriate uniformity conditions when $\rho \neq 1$ or $\rho = 1$.  For
$\rho \leq 1$, the asymptotic equivalences $|R^{(e)}(\bb) - R(\hat{\bb}_r^*,\bb)| \to 0$ and $R^{(e)}(\bb) \sim
R(\hat{\bb}_r^*,\bb)$ follow from
(\ref{jensen}) and (\ref{prop2proof2}); to prove the equivalences for
$\rho > 1$, notice that
\[
\begin{array}{l}\dfrac{1}{n}E\left\{\dfrac{s_1}{s_n}\tr(XX^T +
  d/c^2I)^{-1}\right\} \\
\qquad + 2\dfrac{d-n}{c^2(n-2)}E\tr(XX^T + d/c^2I)^{-2}\end{array}=
O\left\{\frac{||\bb||^2}{n(||\bb||^2 + 1)}\right\}.
\]
Since $R^{(e)}(\bb) = R^{(s)}(||\bb||)$, it suffices to show that
\[
\lim_{d/n\to \rho} \sup_{\bb \in \R^d} \left|R^{(s)}(||\bb||) - R^{(b)}(||\bb||)\right|
= 0
\]
and that $R^{(s)}(||\bb||) \sim R^{(b)}(||\bb||)$ uniformly for $\bb \in
\R^d$ in order to prove the proposition; both follow from (\ref{prop2proof3}).
\end{proof}

Two types of asymptotic equivalence are addressed in Proposition 5:
differences (\ref{prop4a}) and quotients (\ref{prop4b}).  The
equivalence (\ref{prop4a}) is more informative for large $||\bb||$;
(\ref{prop4b}) is more informative for small $||\bb||$.  Notice that
for fixed $||\bb|| = c \in (0,\infty)$, $\lim_{d/n \to \rho}
r_{>0}(d/n,c) = r_{>0}(\rho,c) > 0$ and it follows that (\ref{prop4a}) and
(\ref{prop4b}) are equivalent.  

For $d/n \to 0$, we saw that $\hat{\bb}_{scal}^*$ and $\hat{\bb}_r^*$
were asymptotically equivalent (and that, in some instance, both were
also 
asymptotically equivalent to the OLS estimator; Corollary 1).
When $d/n \to \rho \in (0,\infty)$, 
$\hat{\bb}_r^*$ and $\hat{\bb}_{scal}^*$ are not asymptotically
equivalent.  Indeed, (\ref{ridgejs}) implies that for $d/n \to 0$, we have
\[
R(\hat{\bb}_{scal}^*,\bb) \sim r_{scal}(d/n,||\bb||),
\]
where 
\[
r_{scal}(\rho,c) = \frac{1 - \rho}{1 - \rho + \rho/c^2}.
\]
One easily checks that for $\rho > 0$, $r_{>0}(\rho,c) \leq r_{scal}(\rho,c)$ with
equality if and only if $c = 0$.  Thus, $\hat{\bb}_{scal}^*$ is not
asymptotically minimax over $\ell^2$-balls when $d/n \to \rho \in
(0,\infty)$.  

Despite its suboptimal performance, the estimator $\hat{\bb}_{scal}^*$  
may be useful in certain situations.  Indeed, if $\Cov(\x_i) = \S \neq I$, then it is straightforward to implement a
modified version of $\hat{\bb}_{scal}^*$ with similar
properties (replace $||\bb||^2$ in $\hat{\bb}_{scal}^*$ with
$\bb^T\S\bb$); on the other hand, if $\S$ is unknown and a
norm-consistent estimator for $\S$ is not available, then this may
have a more dramatic effect on the ridge estimator $\hat{\bb}_r^*$.  This is discussed in
detail in \citep{dicker2012dense}, where it is argued that in dense problems where little
is known about $\Cov(\x_i)$, an appropriately modified version of
$\hat{\bb}_{scal}^*$ is a reasonable alternative to ridge regression
(note, for instance, that
$R(\hat{\bb}_{scal}^*,\bb)/R(\hat{\bb}_r^*,\bb) = O(1)$ if $d/n \to
\rho \in (0,\infty)$).

\subsection{$d/n\to \infty$}

Theorem 1 plays a crucial role in our asymptotic analysis when $d/n
\to \rho < \infty$.  It is less relevant in the setting where $d/n \to
\infty$.  Instead, Lemma 1 from Section 2.4 plays the key role.  We
have the following proposition.  

\begin{prop} Suppose that $d > n$ and that $\hat{\bb} \in
\E$.  Then 
\[
R(\hat{\bb},\bb) \geq \frac{d-n}{d}||\bb||^2.
\]
\end{prop}

\begin{proof} Let $X = UDV^T$ be the singular value decomposition of
$X$, as in the proof of Lemma 1. Let $V_0$ and $V_1$ be the first $r$ and the remaining $d-r$
columns of $V$, respectively, where $r = \mathrm{rank}(X)$ (note that
$r = n$ with probability 1).  Then
\begin{eqnarray}
R(\hat{\bb},\bb) & = & E||\hat{\bb} - \bb||^2 \nonumber \\
&  = & E||V_0^T(\hat{\bb} - \bb)||^2  +
E||V_1^T\bb||^2 \label{prop5proof1} \\
& \geq & E||V_1^T\bb||^2 \nonumber \\
& = & \frac{d-n}{n}||\bb||^2, \label{prop5proof2}
\end{eqnarray}
where (\ref{prop5proof1}) follows from Lemma 1 and (\ref{prop5proof2})
follows from symmetry.  \end{proof}

The proof of Proposition 6 essentially implies that for $d >n$, the squared bias
of an equivariant estimator must be at least $||\bb||^2(d-n)/d$.  This
highlights one of the major challenges in high-dimensional dense estimation
problems, especially in settings where $d \gg n$.  The next proposition, which is the main result in this
subsection, implies that if $d/n \to \infty$, then the trivial
estimator $\hat{\bb}_{null} = 0$ is asymptotically minimax. In a
sense, this means that in dense problems $\bb$ is completely non-estimable when $d/n \to
\infty$.  

\begin{prop} Let $\hat{\bb}_{null} =
0$.  Then $R(\hat{\bb}_{null},\bb) = ||\bb||^2$.  Furthermore, if $d/n \to \infty$, then
\[
R(\hat{\bb}_r^*,\bb) \sim R^{(e)}(\bb) \sim R^{(s)}(||\bb||) \sim
R^{(b)}(||\bb||) \sim R(\hat{\bb}_{null},\bb) \sim ||\bb||^2
\]
uniformly for $\bb \in \R^d$.  
\end{prop}

\begin{proof} Clearly, $R(\hat{\bb}_{null},\bb)= ||\bb||^2$.
It follows from Proposition 6 that for $d > n$, 
\[
\begin{array}{l}
\dfrac{d - n}{n} ||\bb||^2  \leq   R^{(e)}(\bb) 
 =  R^{(s)}(||\bb||^2) 
 \leq 
R^{(b)}(||\bb||^2) \\ \qquad \qquad \qquad \qquad \qquad \qquad
 \leq  R(\hat{\bb}_r^*,\bb) 
 \leq R(\hat{\bb}_{null},\bb) 
 =  ||\bb||^2.
\end{array}
\]
The proposition follows by dividing by $||\bb||^2$ and taking $d/n \to
\infty$.  \end{proof}

\subsection{Adaptive estimators}

The results in Section 3.1-3.3 imply that the oracle ridge estimator
$\hat{\bb}_r^* = \hat{\bb}_r(||\bb||)$ is asymptotically minimax over $\ell^2$-balls and
$\ell^2$-spheres and is asymptotically optimal among equivariant
estimators for $\bb$ in any asymptotic setting where $d \to \infty$.
The next result describes asymptotic optimality properties
of the adaptive ridge estimator $\check{\bb}_r^*$ (defined in
(\ref{adapt})), which does not depend on $||\bb||$.

\begin{prop} Suppose that $\rho \in (0,1)$ and let
$R^*(\bb)$ denote any of $R(\hat{\bb}_r^*,\bb)$, $R^{(e)}(\bb)$,
$R^{(s)}(||\bb||)$, $R^{(b)}(\bb)$, or $r_{>0}(d/n,||\bb||)$.  Let
$\{a_n\}_{n = 1}^{\infty} \subseteq \R$ denote a sequence of positive real numbers
such that $a_nn^{1/2} \to \infty$.Then 
\[
\lim_{d/n \to \rho} \sup_{\bb \in \R^d} \left|R^*(\bb) -
  R(\check{\bb}_r^*,\bb)\right| = 0 \mbox{ and } \lim_{d/n \to \rho}
\sup_{||\bb||^2 \geq a_n} \frac{R(\check{\bb}_r^*,\bb)}{R^*(\bb)} = 1.
\]
\end{prop}

Proposition 8 follows immediately from Propositions 3 and 5.  The
restriction $||\bb||^2 \gg n^{1/2}$ in the second part of
Proposition 8 is related to the fact
that for $d/n \to \rho \in (0,\infty)$, $R(\hat{\bb}_r^*,\bb) =
O(||\bb||^2)$ and the error bound in Proposition 3 is $O(n^{-1/2})$.
As discussed in Section 2.5.2, more detailed results on adaptive ridge
estimators are likely possible (that may apply, for instance, in settings where $d/n \to 0$ or $d/n
\to \rho \geq 1$), but this not pursued further here.

\section{An equivalent sequence model}

The rest of the paper is devoted to proving Theorem 1.  In this
section and Section 5, we assume that $d \leq n$.  In Section 6, we
address the case where $d > n$.  The major goal in this section
is to relate the linear model (\ref{lm}) to
an equivalent non-Gaussian sequence model.  

\subsection{The model}

Let $\S$ be
a random orthogonally invariant $m \times m$ positive definite matrix
with rank $m$, almost surely (by orthogonally invariant, we mean that
$\S$ and $U\S U^T$ have the same distribution for any $U \in O(m)$).
Additionally, let $\dd_0 \sim N(0,I_m)$ be a $d$-dimensional Gaussian
random vector that is independent of $\S$.  Recall that in
the sequence model (\ref{seq}), the vector $\z = (z_j)_{j \in J} = \bth
+ \dd$ is observed and $J$ is an index set.  In the formulation
considered here, $J = \{1,...,m\}$, $\dd = \S^{1/2}\dd_0$, and $\S$ is
observed along with $\z$.   
Thus, the available data are $(\z,\S)$ and
\begin{equation}\label{seq0}
\z = \bth + \dd = \bth + \S^{1/2}\dd_0 \in \R^m.
\end{equation}
Notice that $\dd$ is in general non-Gaussian.  However, conditional on
$\S$, $\dd$ is a Gaussian random vector with covariance $\S$.  We are interested in the risk for
estimating $\bth$ under squared error loss.  For an estimator
$\hat{\bth} = \hat{\bth}(\z,\S)$, this is defined by
\[
\tilde{R}(\hat{\bth},\bth) = E_{\bth}||\hat{\bth}(\z,\S) - \bth||^2 = E_{\bth}||\hat{\bth} - \bth||^2,
\]
where the expectation is taken with respect to $\dd_0$ and $\S$ (we
use ``$\sim$,'' as in $\tilde{R}$, to denote quantities related to the
sequence model, as opposed to the linear model).  

\subsection{Equivariance and optimality concepts}
Most of the key concepts initially introduced in the context of the linear model have 
analogues in the sequence model (\ref{seq0}).  In this subsection, we
describe some that will be used in our proof of Theorem 1.  

\vspace{.1in} 
{\em Definition 2.} Let $\hat{\bth} = \hat{\bth}(\z,\S)$ be an
estimator for $\bth$.  Then $\hat{\bth}$ is an {\em
  orthogonally equivariant} estimator for $\bth$ if
\[
U\hat{\bth}(\z,\S) = \hat{\bth}(U\z,U^T\S U) 
\] 
for all $U \in O(d)$.\hfill $\Box$

\vspace{.1in}

Let
\[
\tilde{\E} = \tilde{\E}_d = \{\hat{\bth}; \ \hat{\bth} \mbox{ is an orthogonally
equivariant estimator for } \bth\}
\]
denote the class of orthogonally equivariant estimators for $\bth$.
Also define the posterior mean for $\bth$ under the assumption that $\bth
\sim \pi_c$,
\[
\hat{\bth}_{unif}(c) = E_{\pi_c}(\bth|\z,\S)
\]
and the posterior mean under the assumption that $\bth \sim
N(0,c^2/mI)$.
\[
\hat{\bth}_r(c) = E_{N(0,c^2/mI)}(\bth|\z,\S) =  c^2/d\left\{\S +
  c^2/mI\right\}^{-1}\z
\]
(for both of these Bayes estimators we assume that $\bth$ is independent of
$\dd_0$ and $\S$).
The estimators $\hat{\bth}_{unif}(c)$ and $\hat{\bth}_r(c)$ for $\bth$ are
analogous to the estimators $\hat{\bb}_{unif}(c)$ and
$\hat{\bb}_r(c)$ for $\bb$, respectively.  Moreover, they are both orthogonally
equivariant, i.e. $\hat{\bth}_{unif}(c),\hat{\bth}_r(c) \in
\tilde{\E}$, and $\hat{\bth}_r(c)$ is a linear estimator.  
Now define the minimal equivariant risk for
the sequence model
\[
\tilde{R}^{(e)}(\bth) = \tilde{R}_m^{(e)}(\bth) = \inf_{\hat{\bth} \in \mathscr{E}_{seq}} \tilde{R}(\hat{\bth},\bth)
\]
and the minimax risk over the $\ell^2$-sphere of radius $c$,
\[
\tilde{R}^{(s)}(c) = \tilde{R}_m^{(s)}(c) = \inf_{\hat{\bth}}\sup_{\bth \in
  S(c)} \tilde{R}(\hat{\bth},\bth),
\]
where the infimum above is taken over all measurable estimator
$\hat{\bth} = \hat{\bth}(\z,\S)$.  The Hunt-Stein theorem yields the
following result, which is entirely analogous to Proposition 1.

\begin{prop} Suppose that $||\bth|| = c$.  Then
\[
\tilde{R}^{(s)}(c) = \tilde{R}^{(e)}(\bth) = \tilde{R}\{\hat{\bth}_{unif}(c),\bth\}.
\]
Furthermore, if $\hat{\bth} \in \tilde{\E}$, then
$\tilde{R}(\hat{\bth},\bth)$ depends on $\bth$ only through $c$.  
\end{prop}

\subsection{Equivalence of the sequence model and the linear model}

The next proposition helps characterize the equivalence between the linear
model (\ref{lm}) and the sequence model (\ref{seq0}). 

\begin{prop} Suppose that $d \leq n$, $m = d$, and $\S = (X^TX)^{-1}$.
\begin{itemize}
\item[(a)] If $\bb = \bth$ and $\z = (X^TX)^{-1}X^T\y$, then
  $\hat{\bb}_{unif}(c) = \hat{\bth}_{unif}(c)$,  $\hat{\bb}_r(c) =
  \hat{\bth}_r(c)$, and
\[
R\{\hat{\bb}_r(c),\bb\} = \tilde{R}\{\hat{\bth}_r(c),\bth\}.
\]
\item[(b)] If $||\bth|| = ||\bb|| = c$, then
\[
\begin{array}{l}
R\{\hat{\bb}_{unif}(c),\bb\} = R^{(e)}(\bb) = R^{(s)}(c) \\ \qquad
\qquad \qquad =
\tilde{R}^{(s)}(c) = \tilde{R}^{(e)}(\bth) =
\tilde{R}\{\hat{\bth}_{unif}(c),\bth\}.
\end{array}
\]
\end{itemize}
\end{prop}

Part (a) of Proposition 10 is obvious; part (b) follows from the fact
that $((X^TX)^{-1}X^T\y$, $(X^TX)^{-1})$ is sufficient for $\bb$ and the
Rao-Blackwell inequality. Proposition 10 implies that it
suffices to consider the sequence model in order to prove Theorem 1.

\section{Proof of Theorem 1 (a): Normal approximation for the uniform prior}

It follows from Proposition 9 that the Bayes estimator $\hat{\bth}_{unif}(c)$ is
optimal among all orthogonally equivariant estimators for $\bth$, if
$||\bth|| = c$.  In
this section, we prove Theorem 1 (a) by bounding
\begin{equation}\label{boundthis}
\left|R\{\hat{\bth}_r(c),\bth\} - R\{\hat{\bth}_{unif}(c),\bth\}\right|
\end{equation}
and applying Proposition 10.  

\cite{marchand1993estimation} studied the relationship between
$\hat{\bth}_{unif}(c)$ and $\hat{\bth}_r(c)$ under the
assumption that $||\bth|| = c$ and $\S = \tau^2I$ (i.e. in the
Gaussian sequence model with iid errors).  Marchand
proved the following result, which is one of the keys to the proof of
Theorem 1 (a).

\begin{prop}[Theorem 3.1 from
\citep{marchand1993estimation}]
Suppose that $\S = \tau^2I$ and $||\th|| = c$.  Then 
\begin{eqnarray*}
\left|\tilde{R}\{\hat{\bth}_r(c),\bth\}-\tilde{R}\{\hat{\bth}_{unif}(c),\bth\}\right|
& \leq&\frac{1}{m}\frac{c^2\tau^2m}{c^2 + \tau^2m}
\\
& = & \frac{1}{m}\tilde{R}\{\hat{\bth}_r(c),\bth\}.
\end{eqnarray*}  
\end{prop}

Thus, in the Gaussian sequence model with iid errors, the risk of $\hat{\bth}_r(c)$ is
nearly as small as that of $\hat{\bth}_{unif}(c)$.
Marchand's result relies on somewhat delicate calculations involving
 modified Bessel functions \citep{robert1990modified}.  A direct
 approach to bounding (\ref{boundthis}) for general $\S$
might involve attempting to mimic these calculations.
However, this seems daunting \citep{bickel1981minimax}.
Brown's identity, which relates the risk of a Bayes
estimator to the Fisher, allows us to sidestep these calculations and
apply Marchand's result directly.

Define the Fisher information of a random vector $\bxi \in \R^m$, with
density $f_{\bxi}$ (with respect to Lebesgue measure on $\R^m$) by
\[
I(\bxi) = \int_{\R^d} \frac{\nabla f_{\bxi}(\t)\nabla
    f_{\bxi}(\t)^T}{f_{\bxi}(\t)} \ d\t,
\]
where $\nabla f_{\bxi}(\t)$ is the gradient of $f_{\bxi}(\t)$.
Brown's
identity has typically been used for univariate problems or
problems in the sequence model with iid Gaussian errors
\citep{bickel1981minimax, dasgupta2010false, brown1990information,
  brown1991information}.  The next proposition is a straightforward
generalization to the correlated multivariate Gaussian setting. Its proof is based on Stein's lemma.  

\begin{prop}[Brown's Identity] Suppose that
$\mathrm{rank}(\S) = m$, with probability 1.  Let $I_{\S}(\bth + \S^{1/2}\dd_0)$ denote
the Fisher information of $\bth + \S^{1/2}\dd_0$, conditional on $\S$, under
the assumption that $\bth \sim \pi_c$ is independent of $\dd_0$ and
$\S$.  If $||\bth|| = c$, then
\[
\tilde{R}\{\hat{\bth}_{unif}(c),\bth\} = E\tr(\S) - E\tr\left\{\S^2 I_{\S}(\bth + \S^{1/2}\dd)\right\}.
\]
\end{prop}

\begin{proof} Suppose that $c = ||\bth||$ and let 
\[
f(\z) = \int_{S(c)}
(2\pi)^{-d/2}\det(\S^{-1/2})e^{-\frac{1}{2}(\z -
  \bth)^T\S^{-1}(\z - \bth)} \ d\pi_c(\bth)
\]
be the density of $\z = \bth + \S^{1/2}\dd_0$, conditional on $\S$ and under the assumption that
$\bth \sim \pi_c$.  Then
\[
\hat{\bth}_{unif}(c) = E_{\pi_c}(\bth|\z,\S) = \z - E_{\pi_c}(\S^{1/2}\dd_0|\z,\S) = \z +
\frac{\S\nabla f(\z)}{f(\z)}.
\]
It follows that
\begin{eqnarray}
E||\hat{\bth}_{unif}(c) - \bth||^2& = & 
E\left|\left|\S^{1/2}\dd + \frac{\S\nabla
          f(\z)}{f(\z)}\right|\right|^2 \nonumber \\
& = & E\tr(\S) + 2E\left\{\frac{\dd^T\S^{3/2}\nabla
    f(\z)}{f(\z)}\right\} \nonumber \\
&& \qquad + E\left\{\frac{\nabla m(\z)^T\S^2\nabla
    f(\z)}{f(\z)^2}\right\} \nonumber \\
& = & E\tr(\S) + 2E\left\{\frac{\dd^T\S^{3/2}\nabla
    f(\z)}{f(\z)}\right\} \label{prop6a} \\ && \qquad +E\tr\left\{\S^2 I_{\S}(\bth +
  \S^{1/2}\dd)\right\} \nonumber
\end{eqnarray}
By Stein's lemma (integration by parts),
\begin{eqnarray}\label{prop6b}
E\left\{\frac{\dd^T\S^{3/2}\nabla
    f(\z)}{f(\z)}\right\} & = & E\left[\tr\left\{\S^2 \nabla^2 \log f(\z)
  \right\}\right] \nonumber \\
& = & -E\tr\left\{\S^2 I_{\S}(\bth + \S^{1/2}\dd)\right\}. \label{prop6b}
\end{eqnarray}
Brown's identity follows by combining (\ref{prop6a}) and
(\ref{prop6b}).  
\end{proof}

Using Brown's identity, Fisher information bounds may be converted to risk
bounds, and vice-versa.  Its usefulness in the present context springs
from (i) the decomposition
\begin{equation}\label{decomp}
\z = \bth + \S^{1/2}\dd_0 = \left\{\bth + (\g s_m)^{1/2}\dd_1\right\} + (\S - \gamma s_m)^{1/2}\dd_2,
\end{equation}
where $\dd_1,\dd_2 \iidsim N(0,I_m)$ are independent of $\S$, $s_m$ is the smallest
eigenvalue of $\S$, and $0 < \g < 1$ is a constant and (ii) Stam's inequality for the Fisher information of sums of
independent random variables.

\begin{prop}[Stam's inequality; this version due to \cite{zamir1998proof}] Let $\vv,\w \in \R^m$ be independent random variables that are absolutely continuous
with respect to Lebesgue measure on $\R^m$.   For every $m \times m$ positive definite matrix $\S$, 
\[
\tr\left[\S^2I(\vv + \w)\right] \leq  \tr\left\{\S^2\left[I(\vv)^{-1} + I(\w)^{-1}\right]^{-1}\right\}.
\] 
\end{prop}

Notice that conditional on $\S$, the term $\bth + (\g s_m)^{1/2}\dd_1$ in (\ref{decomp}) may be viewed as
an observation from the Gaussian sequence model with iid errors.  The necessary bound on (\ref{boundthis}) is obtained by piecing together Brown's identity, the
decomposition (\ref{decomp}), and Stam's inequality, so that
Marchand's inequality (Proposition 11) may be applied to $\bth + (\g s_m)^{1/2}\dd_1$.

\begin{prop}  Suppose that $\S$ has rank $m$ with probability
1 and that $||\bth|| = c$.  Let $s_1 \geq \cdots \geq s_m \geq 0$ denote the eigenvalues of
$\S$. Then
\[
\left|\tilde{R}\{\hat{\bth}_r(c),\bth\} - R\{\hat{\bth}_{unif}(c),\bth\}\right| \leq\frac{1}{m}E\left\{\frac{s_1}{s_m}\tr\left(\S^{-1} + m/c^2I\right)^{-1}\right\}.
\]
\end{prop}

\begin{proof} It is straightforward to check that 
\begin{equation}\label{normalRisk}
R\{\hat{\bth}_r(c),\bth\} = E\tr(\S^{-1} +
  m/c^2I)^{-1}.
\end{equation}
Thus, Brown's identity and (\ref{normalRisk}) imply 
\begin{eqnarray*}
\tilde{R}\{\hat{\bth}_r(c),\bth\} - \tilde{R}\{\hat{\bth}_{unif}(c),\bth\}  & = & E\tr\left\{\S^2I_{\S}(\bth +
  \dd)\right\} \\
&& \qquad + E\tr(\S^{-1} +
  m/c^2I)^{-1} - E\tr(\S) \\
& = & E\tr\left\{\S^2I_{\S}(\bth +
  \dd)\right\} \\
&& \qquad - E\tr\left\{\S^2(\S + c^2/mI)^{-1}\right\}.
\end{eqnarray*}
Taking $\vv = \bth + (\g s_m)^{1/2}\dd_1$ and $\w = (\S - \gamma
s_m)^{1/2}\dd_2$ in Stam's inequality, where $\dd_1$, $\dd_2$, and $0
< \g < 1$ are given in (\ref{decomp}), one obtains
\begin{eqnarray*}
\tilde{R}\{\hat{\bth}_r(c),\bth\} - \tilde{R}\{\hat{\bth}_{unif}(c),\bth\} & \leq&
 E\tr\left(\S^2\Big[I_{\S}\{\bth +  (\g s_m)^{1/2}\dd_1\}^{-1}
   \right. \\ && \qquad \left.+ \S - \g s_m I\Big]^{-1}\right) \\
&& \qquad - E\tr\left\{\S^2(\S + c^2/mI)^{-1}\right\}
\end{eqnarray*}
By orthogonal invariance, $ I_{\S}\{\bth +  (\g s_m)^{1/2}\dd_1\}  =
\zeta I_m$ for some
$\zeta \geq 0$.  Marchand's inequality, another application of
Brown's identity, and (\ref{normalRisk}) with $\S = \g s_m I_m$ imply that
\[
\zeta \leq \left(\frac{1}{\g s_m}\right) \frac{\g s_m + c^2/m^2}{\g s_m + c^2/m}. 
\]
Since
\[
\frac{1}{\zeta} - \g s_m \geq (m - 1)\frac{\g
  s_mc^2}{\g s_mm^2 + c^2},
\]
it follows that
\begin{eqnarray*}
\tilde{R}\{\hat{\bth}_r(c),\bth\} - \tilde{R}\{\hat{\bth}_{unif}(c),\bth\} \! & \leq &\! E\tr\left[\S^2\left\{\S
  +  (m - 1)\frac{\g
  s_mc^2}{\g s_mm^2 + c^2}I\right\}^{-1}\right] \\
&& \qquad - E\tr\left\{\S^2(\S+ c^2/mI)^{-1}\right\}.
\end{eqnarray*}
Taking $\g \uparrow 1$,
\begin{eqnarray*}
\tilde{R}\{\hat{\bth}_r(c),\bth\} - \tilde{R}\{\hat{\bth}_{unif}(c),\bth\}  & \leq &  E\tr\left[\S^2\left\{\S
  +  (m - 1)\frac{
  s_mc^2}{s_mm^2 + c^2}I\right\}^{-1}\right] \\
&& \qquad - E\tr\left\{\S^2(\S+ c^2/mI)^{-1}\right\} \\
& \leq & \frac{1}{m}E\left\{\frac{s_1}{s_m}\tr\left(\S^{-1} +
    m/c^2I\right)^{-1}\right\}. \end{eqnarray*}
The proposition follows because
$\tilde{R}\{\hat{\bth}_{unif}(c),\bth\} \leq
\tilde{R}\{\hat{\bth}_r(c),\bth\}$. 
\end{proof}

Theorem 1 (a) follows immediately from Propositions 10 and 14.

\section{Proof of Theorem 1 (b): $d > n$}

It only remains to prove Theorem 1 (b), which is achieved through a
sequence of lemmas.  The
first step of the proof focuses on the linear model (as opposed to
the sequence model) and on reducing the problem where $d > n$ and
$X^TX$ is not invertible to a full rank
problem. This step builds on Lemma 1 from Section 2.4.  

Suppose that $d > n$ and let $X = UDV^T$ be
the singular value decomposition of $X$, where $U \in O(n)$, $V \in
O(d)$, $D = (D_0 \ \ 0)$, and $D_0$ is a rank $n$ diagonal matrix
(with probability 1).  Let $W \in O(d)$ be uniformly distributed on $O(n)$ (according to Haar
measure) and independent of $\ee$ and $X$.  Define the $n \times n$
matrix $X_0 =
UD_0W^T$ and consider the full rank linear model
\begin{equation}\label{lmrk}
\y_0 = X_0\bb_0 + \ee,
\end{equation}
where $\bb_0 \in \R^n$.  Notice that unlike $X$, the entries in $X_0$
are {\em not} iid $N(0,1)$. However, $X_0^TX_0$ is orthogonally
invariant.   As with the linear model (\ref{lm}), one
can consider estimators $\hat{\bb}_0 = \hat{\bb}_0(\y_0,X_0)$ for $\bb_0$ and compute the risk
\begin{equation}\label{riskrk}
R_0(\hat{\bb}_0,\bb_0) = E_{\bb_0}||\hat{\bb}_0 - \bb_0||^2,
\end{equation}
where the expectation in (\ref{riskrk}) is taken over $\ee$ and
$X_0$.  We have the following lemma.

\begin{lemma}
Suppose that $d > n$, $||\bb|| = c$, and $\hat{\bb} \in \E(n,d)$.
Let $P_0$ denote any fixed $n \times d$ projection matrix with
orthogonal rows.  Then there is an orthogonally
equivariant estimator $\P_0\hat{\bb} \in \E(n,n)$ such that
\[
R(\hat{\bb},\bb) = \int_{S_d(c)} R_0(\P_0\hat{\bb},P_0\b) \ d\pi_c(\b) + \frac{d-n}{d}c^2.
\]
\end{lemma}

\begin{proof}
As above, let $X = UDV^T$ be the singular value decomposition of $X$.
Let $V_0$ denote the first $n$ columns of $V$ and let
$V_1$ denote the remaining $d-n$ columns of $V$.  By
(\ref{lemma1proof0}),
\[
\hat{\bb}(\y,X) = V_0\hat{\bb}_0(\y,UD_0),
\]
where $\P_0\hat{\bb}(\y,UD_0) = \hat{\bb}_0(\y,UD_0)$ is the first $n$ coordinates of
$\hat{\bb}(\y,UD)$.  Furthermore, it is easy to check that
$\P_0\hat{\bb}$ is orthogonally equivariant, i.e. $\P_0\hat{\bb} \in
\E(n,n)$.  Thus, 
\begin{eqnarray*}
R(\hat{\bb},\bb) & = &E_{\bb}||\hat{\bb}_0(\y,UD_0) - V_0^T\bb||^2  +
E_{\bb}||V_1^T\bb||^2 \\
& = &E_{\bb}||\hat{\bb}_0(\y,UD_0) - V_0^T\bb||^2  +
\frac{d-n}{d}c^2.
\end{eqnarray*}
To prove the lemma, it suffices to show that
\[
E_{\bb}||\hat{\bb}_0(\y,UD_0) - V_0^T\bb||^2 = \int_{S_d(c)} R_0(\hat{\bb}_0,P_0\b) \ d\pi_c(\b).
\]
By Proposition 1, orthogonal
invariance of $\pi_c$, and orthogonal equivariance of $\hat{\bb}_0$,
\begin{eqnarray*}
E_{\bb}||\hat{\bb}_0(\y,UD_0) - V_0^T\bb||^2 & = & \int_{S_d(c)} E_{\b}||\hat{\bb}_0(\y,UD_0) - V_0^T\b||^2  \
d\pi_c(\b) \\
& = & E\bigg\{\int_{S_d(c)} ||\hat{\bb}_0(UD_0V_0^T\b + \ee,UD_0) \\
&& \qquad \qquad \qquad \qquad \qquad  - V_0^T\b||^2  \
d\pi_c(\b)\bigg\} \\
& = & E\bigg\{\int_{S_d(c)} ||\hat{\bb}_0(UD_0W^TP_0\b + \ee,UD_0) \\
&& \qquad \qquad \qquad \qquad  \ \  - W^TP_0\b||^2  \
d\pi_c(\b)\bigg\} \\
& = & \int_{S_d(c)} E||\hat{\bb}_0(\y_0,X_0) - P_0\b||^2  \
d\pi_c(\b),
\end{eqnarray*}
as was to be shown.  
\end{proof}

Lemma 2 allows us to express the risk of an equivariant estimator for
$\bb$ in the linear model (\ref{lm}) with
$d> n$ in terms of the risk of another equivariant estimator in a
different linear model (\ref{lmrk}) with $d = n$.  Though the
linear model (\ref{lmrk}) differs from the original linear model with
Guassian predictors -- thus, Theorem 1 (a) does not apply directly --
(\ref{lmrk}) is equivalent to the sequence model (\ref{seq0}), with $m
= n$ and $\S = (X_0^TX_0)^{-1}$.  

\begin{lemma}
Suppose that $2 < m = n < d$ and $\S = (X_0^TX_0)^{-1}$ in the sequence
model (\ref{seq0}).  Also suppose that $||\bb|| = c$. Let $P_0$ be a fixed $n \times d$ projection matrix with orthogonal
rows and let $s_1 \geq \cdots \geq s_n \geq 0$ denote the eigenvalues
of $(X^TX)^{-1}$.   Then
\begin{eqnarray*}
R\{\hat{\bb}_{unif}(c),\bb\} & \geq &  \int_{S_d(c)}
\tilde{R}\{\hat{\bth}_{unif}(P_0\t),P_0\t\} \ d\pi_c(\t) +
\frac{d-n}{d}c^2  \\
& \geq & \int_{S_d(c)} \!\!\!
E\left\{\left(1 - \frac{s_1}{ns_n}\right)\tr\left(XX^T +
  \frac{n}{||P_0\t||^2}I\right)^{-1}\right\} \ \! d\pi_c(\t) \\
&& \qquad + \frac{d-n}{d}c^2 \\
& \geq & \! \!E\left[\left(1 - \frac{s_1}{ns_n}\right)\tr\left\{XX^T +
  \frac{n(d-2)}{c^2(n-2)}I\right\}^{-1}\right] + \frac{d-n}{d}c^2.
\end{eqnarray*}
\end{lemma}

\begin{proof}
The first inequality follows from Lemma 2 and a suitably modified
version of Proposition 10 that describes the equivalence between the
linear model (\ref{lmrk}) and the sequence model (\ref{seq0}).  The second inequality follows from Proposition 14 and the fact that
$X_0^TX_0$ and $XX^T$ have the same eigenvalues:
\begin{eqnarray*}
\tilde{R}\{\hat{\bth}_{unif}(P_0\t),P_0\t\}& \geq & 
\tilde{R}\{\hat{\bth}_r(P_0\t),P_0\t\}  \\
&& \qquad - \frac{1}{n}E\left\{\frac{s_1}{s_n}\tr(X_0X_0^T +
  n/||P_0\t||^2I)^{-1}\right\} \\
& = & E\left\{\left(1 - \frac{s_1}{ns_n}\right)\tr\left(X_0^TX_0 +
  n/||P_0\t||^2I\right)^{-1}\right\} \\
& = & E\left\{\left(1 - \frac{s_1}{ns_n}\right)\tr\left(XX^T +
  n/||P_0\t||^2I\right)^{-1}\right\}.
\end{eqnarray*}  
The last inequality in the lemma follows from Jensen's inequality and
the identity
\[
\int_{S_d(c)} \frac{1}{||P_0\t||^2} \ d\pi_c(\t) = \frac{d-2}{c^2(n-2)}.
\]  
\end{proof}

We now have  the tools to complete the proof of Theorem 1
(b).  Suppose that $d >n$ and $||\bb|| = c$.  Then
\[
R(\hat{\bb}_r^*,\bb) = E\tr\{XX^T + d/c^2I\}^{-1} +
\frac{d-n}{d}c^2.  
\]
Since $R\{\hat{\bb}_r(c),\bb\} - R\{\hat{\bb}_{unif}(c),\bb\} =
R\{\hat{\bb}_r(c),\bb\} -R^{(e)}(\bb) \geq 0$, Lemma 3 implies 
\begin{eqnarray*}
\left|R\{\hat{\bb}_r(c),\bb\} -R^{(e)}(\bb)\right| & \leq & E\tr\{XX^T +
d/c^2I\}^{-1} \\
&& \ - E\left[\left(1 - \frac{s_1}{ns_n}\right)\tr\left\{XX^T +
  \frac{n(d-2)}{c^2(n-2)}I\right\}^{-1}\right] \\
& \leq & \frac{1}{n}E\left\{\frac{s_1}{s_n}\tr(XX^T +
  d/c^2I)^{-1}\right\} \\
&& \qquad + 2\frac{d-n}{c^2(n-2)}E\tr(XX^T + d/c^2I)^{-2}.
\end{eqnarray*}
Theorem 1 (b) follows.  

\section*{Acknowledgements}

The author thanks Sihai Zhao for his thoughtful comments and
suggestions.

\end{document}